\documentclass[12pt]{amsart}
\usepackage[latin1]{inputenc}
\usepackage{mathptmx}
\usepackage{tikz}
\usepackage{enumerate}
\usepackage{hyperref}

\newtheorem{theorem}{Theorem}[section]
\newtheorem{lemma}[theorem]{Lemma}

\newtheorem{proposition}[theorem]{Proposition}
\newtheorem{remark}[theorem]{Remark}

\theoremstyle{definition}

\newtheorem{example}[theorem]{Example}
\newtheorem{remark/example}[theorem]{Remark/Example}

\let\oldlabel=\label
\def\prellabel{\marginparsep=1em\marginparwidth=44pt
 \def\label##1{\oldlabel{##1}\ifmmode\else\ifinner\else
 \marginpar{{\footnotesize\ \\ \tt
 ##1}}\fi\fi}}

\prellabel

\numberwithin{equation}{section}

\def\PP{ {\bf P} }
\def\NN{ {\bf N} }
\def\ZZ{ {\bf Z} }
\def\QQ{ {\bf Q} }

\def\aa{{\bf a}}
\def\bb{{\bf b}}

\def\O{\mathcal O}
\def\V{\mathcal V}
\def\W{\mathcal W}
\def\T{\mathcal T}

\newcommand{\GCD}{\operatorname{GCD}}

\newcommand{\HS}{\operatorname{HS}}

\newcommand{\Ext}{\operatorname{Ext}}
\newcommand{\Hom}{\operatorname{Hom}}
\newcommand{\Sym}{\operatorname{Sym}}

\numberwithin{equation}{section}

\begin{document}
\sloppy

\title{A remark on hyperplane sections of rational normal scrolls}
\author{Aldo Conca}
\address{ Dipartimento di Matematica,
Universit\`a degli Studi di Genova, Italy} \email{conca@dima.unige.it}
\author{Daniele Faenzi}
\address{Institut de Math\'ematiques de Bourgogne,
UMR CNRS 5584,
Universit\'e de Bourgogne,
9 avenue Alain Savary -- BP 47870,
21078 DIJON CEDEX, France}
\email{daniele.faenzi@u-bourgogne.fr}

\subjclass[2000]{ }
\date{}
\maketitle
 
 \section{Introduction} The degree $\deg X$ of an irreducible non-degenerate projective variety $X\subset \PP^n$ over an algebraically closed field satisfies $$\deg X\geq 1+n-\dim X$$ and the variety $X$ is said to have minimal degree (or minimal multiplicity) if  $\deg X=1+n-\dim X$. 

Projective varieties of minimal degree are completely  classified by a famous result of Bertini and Del Pezzo, see the centennial account of Eisenbud and Harris  \cite{EH}.  They are: 
 
 \begin{enumerate}[1)] 
 \item quadric hypersurfices,
 \item the (quadratic) Veronese embedding of $\PP^2\to \PP^5$,
 \item the rational normal scrolls,  
 \end{enumerate} 
 or cones over them. 

A rational normal scroll of dimension $d$ is associated to a sequence of positive integers $a_1,\dots, a_d$ as it is explained,  for example in \cite{EH} or \cite{H}. We will denote by $S(a_1,\dots,a_d)$  the rational normal scroll associated with $a_1,\dots, a_d$.

Let $a_1,\dots, a_d$ be positive integers and let  $X=S(a_1,\dots,a_d)$ be the associated rational normal scroll with  $d>1$. Consider an hyperplane section $Y$ of  $X$ and assume $Y$  is irreducible. Hence $Y$ is an irreducible variety of minimal degree.  One can easily exclude that $Y$  is the Veronese surface in $\PP^5$ and hence it must be a rational normal scroll of dimension $d-1$. Therefore there exist integers $b_1,\dots,b_{d-1}$ such that  $Y=S(b_1,\dots,b_{d-1})$. How are the numbers  $a_1,\dots, a_d$ and $b_1,\dots,b_{d-1}$ related? This is the question we want to address. Indeed we present algebraic and geometric arguments  that give  a complete classification of the rational normal scrolls  that are hyperplane section of  a given rational normal scrolls. 

\section{Algebraic formulation}
Let us first formulate the question and discuss the  solution in algebraic terms. Consider the polynomial rings 
$$A=K[x,y]=\bigoplus _{i\in \NN} A_i$$ and 
$$B=A[s_1,\dots,s_d]$$ where $K$ is a field and $x,y,s_1,\dots,s_d$ are indeterminates.  
 The coordinate ring of  $S(a_1,\dots,a_d)$ is 
$$R(a_1,\dots,a_d)=K[A_{a_1}s_1,\dots, A_{a_d}s_d]\subset B.$$ 
We introduce a $\ZZ^2$-graded structure in $B$ by giving degree $(1,-a_i) $ to  $s_i$ and degree $(0,1)$ to $x$ and $y$. 
In this way,  $R(a_1,\dots,a_d)$ is identified  with   $\oplus_{j\geq 0} B_{(j,0)}$. An element of degree $1$ in $R(a_1,\dots,a_d)$ is an element of $B_{(1,0)}$ and hence has the form: 
 $$L=f_1s_1+\dots+f_ds_d$$
 where $f_i\in A_{a_i}$ for $i=1,\dots,d$. First we note: 

\begin{lemma}
\label{basic1}  Let $I=(f_1,\dots,f_d)\subset A$, $L=f_1s_1+\dots+f_ds_d $ and assume $L\neq 0$. Then  
the ideal $(L) \subset R(a_1,\dots,a_d)$ is prime if and only if  $A/I$ is Artinian i.e. the ideal  $I$ has codimension $2$ in $A$.  \end{lemma}
\begin{proof} Set $E=R(a_1,\dots,a_d)$. Note that $E$  is a direct summand (as $E$-module) of $B$ and hence $JB \cap E=J$ for every ideal $J$ of $E$. In particular $(L)E=(L)B\cap E$ and therefore $(L)E$  is prime in $E$    if  $(L)B$ is prime in $B$. Viceversa,  if $L$ factors in $B$ the one can easily play with the factors and show that  $(L)E$ is not prime. This show that  $(L)E$  is prime in $E$    if  and only  $(L)B$ is prime in $B$ if and only if $L$ is irreducible in $B$.  The only possible factorizations of $L$ are of type $L=g(f'_1s_1+\dots+f'_ds_d)$ with $g,f_1',\dots, f_s'\in A$ and the conclusion follows.  
\end{proof}

Consider the graded homomorphism of $A$-modules  
\begin{eqnarray}
\label{map}
\Phi:F=\bigoplus_{i=1}^d  A(-a_i) \to A
\end{eqnarray} 
with $\Phi(e_i)=f_i$. If at least one of the $f_i$'s is non-zero then the kernel of $\Phi$ is free of rank $d-1$, say isomorphic to 
$$G=\bigoplus_{i=1}^{d-1}  A(-b_i).$$

\begin{lemma} 
If  $I=(f_1,\dots, f_{d})$ has codimension $2$ in $A$  then 
$$a_1+\dots+a_d=b_1+\dots+b_{d-1}.$$
\end{lemma}
\begin{proof} 
Using the data of the (possibly non-minimal) resolution

\begin{eqnarray}
\label{res}
0\to G\to F\to A\to 0
\end{eqnarray} 
of $A/I$ one gets the following expression for the Hilbert series of $A/I$: 

$$\HS(A/I,z)=\frac{1-\sum_{i=1}^d  z^{a_i}+\sum_{i=1}^{d-1}  z^{b_i}}{(1-z)^2}$$
Since $\HS(A/I,z)$ is a polynomial, the first derivative of the numerator must vanish at $1$. This gives the desired result. 
\end{proof}

With the notation above we have: 

 \begin{theorem}
 \label{basic2}   With the notation above assuming that $(L)$ is prime we have
 $$R(a_1,\dots,a_d)/(L)\simeq R(b_1,\dots,b_{d-1})$$
 as graded $K$-algebras. 
\end{theorem}
\begin{proof} 
Set $I=(f_1,\dots,f_d)$. Applying $\Hom(-, A)$ to (\ref{res}) we have: 

\begin{eqnarray}
\label{res1}
 0\to A^* \to F^* \to G^*\to 0
\end{eqnarray} 

where $A^*=A$ and 
 $$F^*=\Hom(F,A)=\bigoplus_{i=1}^d  A(a_i)$$  and 
 $$G^*=\Hom(G,A)=\bigoplus_{i=1}^{d-1}  A(b_i).$$
Since $I$ has codimension $2$, the complex  (\ref{res1}) has homology only in position $2$ (at $G^*$) and it is, by definition, 
$\Ext_A^2(A/I,A)$. 

Denoting by $s_1,\dots,s_d$ the basis elements of $F^*$, then  the symmetric algebra $\Sym_A(F^*)$ of $F^*$ (view as a $A$-module) can be  identified with the algebra 
$B=A[s_1,\dots,s_d]$. It is naturally a $\ZZ^2$-graded  $K$-algebra with grading introduced above, that is:  
$$
   \deg x=(0,1), \\ \deg y=(0,1)  \\ \deg s_1=(1, -a_1),\dots, \deg s_d= (1, -a_d),
$$
where the first index  identifies the corresponding  symmetric power and the second index is the internal degree. 
 Similarly, denoting by $t_1,\dots,t_{d-1}$ the basis elements of $G^*$  the symmetric algebra $\Sym_A(G^*)$ of $G^*$ can be identified with the polynomial ring $C=A[t_1,\dots,t_{d-1}]$ that is $\ZZ^2$-graded  by 
 $$ 
   \deg x=(0,1), \\ \deg y=(0,1)  \\ \deg t_1=(1, -b_1),\dots, \deg t_{d-1}= (1, -b_{d-1}).
$$

Moreover the  image of the map $A^* \to F^*$ is generated by  $L=f_1s_1+\dots+f_ds_d$. The map  $F^* \to G^*$ induces a map of symmetric algebras: 
$$B=\Sym_A(F^*) \to C=\Sym_A(G^*)$$ 
and, since $L$ is in the kernel of  $F^* \to G^*$, we have an induced $\ZZ^2$-graded $K$-algebra map: 
$$B/(L) \to C$$
(not surjective in general). Taking on both sides the subalgebra of the elements of degree $(*,0)$ we obtain a $\ZZ$-graded $K$-algebra map: 
\begin{eqnarray}
\label{omo}
\bigoplus_{i\in \NN}   [ B/(L) ]_{(i,0)} \to \bigoplus_{i\in \NN} C_{(i,0)}.
\end{eqnarray} 
 Now 
 $$\bigoplus_{i\in \NN}   [ B/(L) ]_{(i,0)}= [\bigoplus_{i\in \NN}  B_{(i,0)}] /(L) $$
 because $L$ has degree $(1,0)$. 
 Observe that $\bigoplus_{i\in \NN}  B_{(i,0)}$ is $K[A_{a_1}s_1,\dots, A_{a_d}s_d]$, that is, $R(a_1,\dots,a_d)$. Similarly   $\bigoplus_{i\in \NN} C_{(i,0)}$ is $K[A_{b_1}t_1, \dots, A_{b_{d-1}}t_{d-1}]$,  that is $R(b_1,\dots,b_{d-1})$.  
Therefore because of (\ref{omo}) we have a  $\ZZ$-graded $K$-algebra map: 
\begin{eqnarray}
\label{iso}
R(a_1,\dots,a_d)/(L) \to R(b_1,\dots,b_{d-1}).
\end{eqnarray} 
Both the rings involved in (\ref{iso}) are domains of Krull dimension $d$. Hence to prove that  (\ref{iso}) is an isomorphism, it is enough to prove that it is surjective. Being   standard graded $K$-algebras, it is enough to prove that  (\ref{iso}) is surjective in degree $1$ (i.e. degree $(1,0)$). Therefore it is enough to prove that the original map $F^*\to G^*$ is surjective in degree $0$, equivalently that $\Ext_A^2(A/I,A)_0=0$. 
Let $g,h$ be a regular sequence in $I$ of degree $u,v$ and set $J=(g,h)$. Then by the graded version of  \cite[Lemma 1.2.4]{BH} the module $\Ext_A^2(A/I,A)$ can be identified with 
$(J:I/J)(-u-v)$. It follows that $$\Ext_A^2(A/I,A)_j=0 \mbox{ for } j \geq -1.$$
\end{proof}

\begin{remark}
\label{rema} 
Given positive integers $a_1,\dots,a_d$ and $X=S(a_1,\dots,a_d)$ we have seen that the following conditions are equivalent: 

\begin{enumerate}[1)]
 \item $S(b_1,\dots,b_{d-1})$  is an hyperplane section of $X$. 
 \item There exists $f_1\in A_{a_1},\dots, f_d \in A_{a_d}$ such that the ideal $(f_1,\dots,f_d)$ has codimension $2$ and the kernel of  (\ref{map}) is generated by elements of  degree $b_1,\dots,b_{d-1}$. 
 \end{enumerate} 
 \end{remark}

  We give now a numerical characterization of the sequences $b_1,\dots,b_{d-1}$ satisfying the equivalent conditions of (\ref{rema}).

\begin{proposition} \
\label{bseq} Assume $a_1\leq a_2\leq \dots \leq a_d$. Then a sequence $b_1\leq \dots \leq b_{d-1}$  of   positive integers 
 satisfies  the equivalent conditions described in \ref{rema}  if and only if: 
\begin{enumerate}[(i)] 
\item \label{uno}  $a_1+\dots+a_d=b_1+\dots+b_{d-1}$,
\item \label{due} $a_j\leq b_j$ for every   $j=1,\dots,d-1$, 
\item \label{tre} Let $v=\min\{ j : a_j<b_j\}$. Then $b_j\geq a_{j+1}$ for every $j\geq v$.   
 \end{enumerate}
 \end{proposition} 
 
 \begin{proof} First we prove that condition (2) of \ref{rema} implies (i)--(iii). Indeed (i) is already established. If  $f_1,\dots,f_d$ are minimal generators (of the ideal they generate) then  $a_{j+1}<b_j$ for every $j=1,\dots,d-1$ otherwise one of the maximal minor of the syzygy matrix will be $0$ contradicting the Hilbert-Burch theorem. Hence in that case (ii) and (iii) hold.    
 Any resolution is obtained from a minimal one by adding copies of the trivial complex $0\to A(-c)\to A(-c)\to 0$. 
 Hence it is enough to prove that if conditions (i)--(iii) hold for  $\aa=a_1\leq a_2\leq \dots \leq a_d$ and $\bb=b_1\leq \dots \leq b_{d-1}$ then they still hold also if we insert a given positive number $c$ in both $\aa$ and $\bb$. To do this we denote by  $\aa'$ and $\bb'$ the weakly increasing sequences the  obtained from  $\aa$ and $\bb$ by inserting $c$. We  distinguish two cases:  .  
 
 \medskip 
 
\noindent  Case (1)  $c< a_v$. Consider the smallest $u$ such that  $c <a_u$. If $u=1$ then $\aa'=c,\aa$ and  $\bb'=c,\bb$ that clearly satisfies the conditions  (i)--(iii). If $u>1$ then $a_{u-1}=b_{u-1}<c\leq a_u\leq b_u$ and hence 
 $$
 \begin{array}{ll}
 \aa'=a_1,\dots, a_{u-1}, c, a_u,\dots, a_{d-1}, a_{d} \\ 
 \bb'=b_1,\dots, b_{u-1}, c, b_u,\dots, b_{d-1}
 \end{array} 
 $$ 
 that clearly satisfies the conditions  (i)--(iii).  
 
\medskip 
 
\noindent  Case (2)  $c\geq a_v$. If $c\geq b_{d-1}$    then $c\geq a_d$ as well and $\aa'=\aa,c$ and  $\bb'=\bb,c$ that clearly satisfies the conditions  (i)--(iii). If $c<b_{d-1}$ let $u$ be smallest such that $c<b_u$.  If $u>v$ then we have  $b_{u-1}\leq c<b_u$  and $a_u\leq  b_{u-1}\leq c$. Let $t$ be the largest index with  $a_t\leq c$ ($t$ might be $d+1$). We have that $t\geq u$ and hence

$$ 
 \begin{array}{lllllllll}
 \aa'= \dots, a_{u-1}, & a_u, &   a_{u+1},& \dots, &a_t ,     & c,   &  a_{t+1}, \dots, a_d, a_{d+1}  \\ 
 \bb'= \dots, b_{u-1}, & c ,   &     b_u,    &  \dots, & b_{t-1}, &b_t, & b_{t+1}, \dots, b_{d}
 \end{array} 
$$
that clearly satisfies the conditions  (i)--(iii). One argues similarly in the remaining  case $u\leq v$. 
 
 Now we show that if conditions (i)--(iii) hold then   (2) of \ref{rema} holds. We set 
 $$\begin{array}{llll} 
&  \alpha_j=b_j-a_j & \mbox{ for }  & j=1,\dots, d-1  \\
 \mbox{ and }\\
 & \beta_j=b_j-a_{j+1}  & \mbox{ for } &  j=v,\dots, d-1 
\end{array}
$$
that, by assumption,  are all non-negative integers.  Consider the $d\times (d-1)$ matrix $Z=(z_{ij})$ whose entries are all $0$ with the exception of 
 
 $$\begin{array}{llll}
 & z_{jj}=y^{\alpha_j} & \mbox{ for } &  j=1,\dots, d-1 \\ \mbox{ and }  \\
 & z_{j+1,j}=x^{\beta_j}  & \mbox{ for } &  j=v,\dots, d-1.
 \end{array} 
$$
Then by constructions the maximal minors of $Z$ are either $0$ or  monomials  of  degree $a_1,\dots,a_d$ and the columns of the $Z$ generate their syzygy module. Keeping track of the degrees one checks that  the generators have exactly  degree $b_1,\dots,b_{d-1}$.
 \end{proof}

  \begin{example} 
  For example, if $d=6$ and $\aa=9, 10, 11, 11, 14, 14$ then 
 $\bb=9, 13, 13, 14, 20$ satisfies the conditions  (i)--(iii) of \ref{bseq} and $v=2$. 
$$
 \begin{array}{cccccccccccc}
\aa=& 9& \leq  & 10 & \leq & 11 & \leq &   11 & \leq  & 14 & \leq & 14\\
    &     ||   &     &  &     & &    & &    & &     \\
\bb=&  9& \leq  &   13& \leq  &  13& \leq  &  14& \leq  &  20
 \end{array}
 $$
 The matrix $Z$ is 
$$Z=\left( 
\begin{array}{ccccc} 
  1& 0& 0& 0& 0  \\
  0& y^3& 0& 0& 0 \\
  0& x^2& y^2& 0& 0 \\
  0& 0& x^2& y^3& 0 \\
  0& 0& 0& x^0& y^6 \\
  0& 0& 0& 0& x^6 
\end{array} 
\right) 
$$
and the the polynomials are
$$f_1=0,\  f_2=x^{10},\   f_3=x^8y^3,\   f_4=x^6y^5,\   f_5=x^6y^8,\      f_6=y^{14}. $$
 \end{example}

\section{A geometric point of view}

Let us reformulate the results in a geometric language. Given integers
$0\le a_1\le \ldots \le a_d$, define the vector bundle  over $\PP^1$
\[
\V = \bigoplus_{i=1}^d \O_{\PP^1}(a_i).
\]
Over the projective bundle $P=\PP(\V)$ we have a tautological relatively
ample line bundle which we denote by $\O_P(\xi)$. 
We identify: 
\begin{equation}
  \label{identificazione}
H^0(P,\O_P(\xi)) \simeq \bigoplus_{i=1}^d H^0(\PP^1,\O_{\PP^1}(a_i)).  
\end{equation}
Let us call $V$ the dual of the above vector space and $\pi$ the
projection $P \to \PP^1$.

Let $f$ be the morphism associated with the linear system
$|\O_P(\xi)|$. In view of \eqref{identificazione}, we write
\[
f \colon P \to \PP(V),
\]
and the image of $f$ is the variety $X$.
The morphism $f$ is birational onto its image. It is actually a closed embedding if and
only if $a_1 > 0$. We assume $a_1>0$ in this section.

Given   integers $b_1\le \cdots \le b_{d-1}$, define
\[
\W \simeq \bigoplus_{i=1}^{d-1} \O_{\PP^1}(b_i).
\]
Set $\O_{\PP(W)}(\eta)$ for the tautological relatively
ample line bundle over $\PP(\W)$ and $g$ for the morphism associated
to this line bundle.
The counterpart of Theorem \ref{basic2} is:

\begin{theorem} \label{geo}
  Given an irreducible hyperplane section $Y$ of $X$, there are
  integers $b_1\le \cdots \le b_{d-1}$ satisfying the 
  conditions of Proposition \ref{bseq} such that
  $Y$ is the image of $g$.
\end{theorem}

\begin{proof}
The surjection $V\otimes \O_P \to \V$ induces a closed embedding
$\PP(\V) \subset \PP^1 \times \PP(V)$ and the map $f$ is just the
composition of this embedding with the projection to the second factor.

A hyperplane section $Y$ of $X$ is determined by a non-zero global section $s$ of
$\O_P(\xi)$, where we think the hyperplane $H$ of $\PP(V)$ as $\PP(V/\langle s \rangle)$. By \eqref{identificazione}, the section $s$ corresponds to a map
\[
\O_{\PP^1} \to  \V.
\]

Write $\W_0$ for the cokernel of this map.
The section $s$ gives a surjection $(V/\langle s \rangle)
\otimes \O_{\PP^1} \to \W_0$ and therefore a morphism $g:\PP(\W_0) \to
H$ which is induced by $f : \PP(\V) \to \PP(V)$ upon restriction with
$H$. In other words, $Y=X \cap H$ is the image via $g$ of $\PP(\W_0)$
and $g=f|_{\PP(\W_0)}$ is associated with the line bundle
$\O_{\PP(\W_0)}(\eta_0)$, $\eta_0$ being the relatively ample
tautological line bundle of $\PP(\W_0)$.

Observe that, in order for $\PP(\W_0)$ to be irreducible, $\W_0$ has to
be torsionfree (and hence locally free). Indeed, let $\T$ be the maximal torsion subsheaf of
$\W_0$ and put $\W_1=\W_0/\T$. Since $\W_1$ is
locally free, we have $\Ext^1_{\PP^1}(\W_1,\T)=0$ so $\W_0 = \T \oplus
\W_1$. Then the projection $\W_0 \to \W_1$ gives an embedding $\PP(\W_1)
\subset \PP(\W_0)$ which shows that $\PP(\W_1)$ is the main component
of $\PP(\W_0)$. So we must have $\T=0$ for $\PP(\W_0)$ to be irreducible. 

We have thus proved that $\W_0 \simeq \W$ (and thus $\eta = \eta_0$)
for some integers $b_1\le \cdots \le
b_{d-1}$.
Then we look at the exact sequence
\[
0 \to \O_{\PP^1} \xrightarrow{s} \bigoplus_{i=1}^{d} \O_{\PP^1}(a_i) \xrightarrow{t} \bigoplus_{i=1}^{d-1} \O_{\PP^1}(b_i)
\to 0
\]

Condition \eqref{uno} of Proposition
\ref{bseq} is verified by computing the total first Chern
class. Condition \eqref{due} is clear, since otherwise the (lower
triangular matrix associated with the) map $t$
could not be surjective.
Also, if $b_j < a_{j+1}$ for some $j$, then the only summands of $\V$ mapping
to $\oplus_{i=1}^j\O_{\PP^1}(b_i)$ add up to
$\oplus_{i=1}^j\O_{\PP^1}(a_i)$. Since these two bundles have the same
rank and the restriction of $t$ is a surjective map among them, we get that
this map is an isomorphism so $a_i=b_i$ for all $i
\le j$. This proves condition \eqref{tre}.
\end{proof}

 \section{Generic case and examples}

\subsection{Generic case}
 
 What is the general hyperplane section of $S(a_1,\dots,a_d)$? 
This question is answered in \cite{B}. A different combinatorial description  of the solution to the problem is obtained by using Fr\"oberg's  characterization of generic Hilbert functions in $A=K[x,y]$.  

Given a formal power series $c(z)=\sum_{j\geq 0} c_i z^i\in \QQ[|z|]$ one sets 
$$\left [ c(z) \right ]_+= \sum_{j\geq 0}c'_i z^i\in \QQ[|z|]$$
where 
$c'_i=c_i$ if $c_j> 0$  for all $0\leq j\leq i$ and $c'_i=0$ otherwise. 
According to Fr\"oberg's result \cite{F},  given numbers $a_1,\dots,a_d$ and the ideal $I=(f_1,\dots,f_d)$ generated by general polynomials with $\deg f_i=a_i$, the Hilbert series of $A/I$ is given by: 
$$\left [ \frac{\prod_{i=1}^d (1-z^{a_i})}{(1-z)^2} \right ]_+$$
Hence the degrees $b_1,\dots,b_{d-1}$  of the syzygies of $f_1,\dots,f_d$ are obtained by the following formula: 
\begin{eqnarray}\label{gen}
\sum_{i=1}^{d-1} z^{b_i}=(1-z)^2\left [ \frac{\prod_{i=1}^d (1-z^{a_i})}{(1-z)^2} \right ]_+  +\sum_{i=1}^d z^{a_i}-1. 
\end{eqnarray} 

\subsection{Examples} Let us give some examples of the possible
irreducible linear sections of a specific 4-fold of minimal degree.

\begin{example} 
\label{gengen}
For example, if $d=4$ and $(a_1,a_2,a_3,a_4)=(4,5,6,9)$ then the Hilbert series of $A/I$ where $I$ is generated by general polynomials of degrees $4,5,6,9$  is 
\begin{align*}
&\left [ \frac{\prod_{i=1}^d (1-z^{a_i})}{(1-z)^2} \right ]_+=\\
&\left [1+2z+3z^2+4z^3+4z^4+3z^5+z^6-z^7\dots \right]_+=\\
&1+2z+3z^2+4z^3+4z^4+3z^5+z^6
\end{align*}
and applying (\ref{gen}) we obtain: 
\begin{align*}
&\sum_{i=1}^{d-1} z^{b_i}=\\
&(1-z)^2(1+2z+3z^2+4z^3+4z^4+3z^5+z^6) +z^4+z^5+z^6+z^9-1=\\
&z^7 + z^8 + z^9,
\end{align*}
that is, the degrees of the syzygies are $(7,8,9)$. In other words, the generic hyperplane section of $S(4,5,6,9)$ is $S(7,8,9)$. According to \ref{bseq} the rational normal scroll $S(4,5,6,9)$ has $15$ other (non-generic) irreducible hyperplane sections that correspond to the following sequences: 

$$
\begin{array}{lllll}
(4, 5, 15)   &  (4, 6, 14) & (4, 7, 13)  & (4, 8, 12) & (4, 9, 11) \\
(4, 10, 10) & (5, 6, 13) & (5, 7, 12)  & (5, 8, 11)  & (5, 9, 10) \\
(6, 6, 12)   &   (6, 7, 11) & (6, 8, 10) & (6, 9, 9) &  (7, 7, 10).
 \end{array} 
$$
\end{example} 

 \begin{example} Table of specializations for scrolls of codimension $5$  
 where we denote in red/dashed   the generic hyperplane sections and in blue/continuos  the non-generic hyperplane sections. 

$$
 
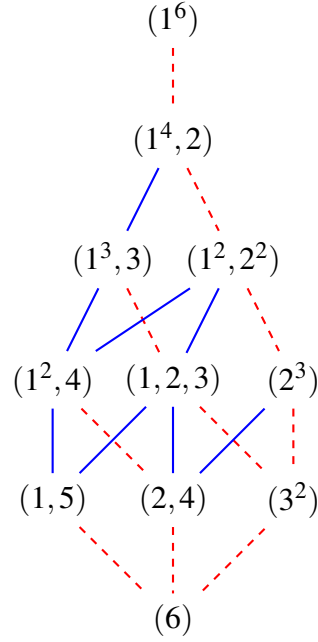
\begin{tikzpicture}[scale=.8]

\node (51) at (0,10)      {$(1^6)$};

\node (41) at (0,8)      {$(1^4,2)$};

 \node (32) at (1,6)      {$(1^2,2^2) $};
 \node (31) at (-1,6)     {$(1^3,3)$};
 
  \node (23) at (2,4)       {$(2^3)$};
  \node (22) at (0,4)      {$(1,2,3)$};
   \node (21) at (-2,4)     {$(1^2,4)$};

 \node (13) at (2,2)      {$(3^2)$};
  \node (12) at (0,2)      {$(2,4)$};
   \node (11) at (-2,2)      {$(1,5)$};
   
    \node (0) at (0,0)      {$(6)$};

\draw [red, dashed, thick] (51) -- (41) -- (32) -- (23) -- (13)     ;
\draw [blue, thick] (41)  -- (31) -- (21) --(11)     ;
\draw [red, dashed,thick] (31) -- (22) -- (13)    ;
\draw [blue, thick] (32)  -- (22) -- (11)    ;
\draw [blue, thick] (32)  -- (21)     ;
\draw [red, dashed,thick] (21) -- (12) ; 
\draw [blue, thick] (22)  -- (12)     ;
\draw [blue, thick] (23)  -- (12)     ;
\draw [red, dashed,thick] (11) -- (0) ; 
\draw [red, dashed, thick] (12) -- (0) ; 
\draw [red, dashed, thick] (13) -- (0) ;
\end{tikzpicture}
$$
 \end{example}

\section{Cones and reducible sections}

We have now a rather complete knowledge of the behavior of irreducible
hyperplane sections of smooth varieties of minimal degree. So what
about reducible ones? What about singular varieties? Here we answer to
these two questions.

\subsection{Reducible hyperplane sections}

Take again $X = S(a_1,\ldots,a_d) \subset \PP^n$ for some  $1 \le a_1
\le \cdots \le a_d$.
\begin{theorem}
\label{redusect}
 Given a hyperplane section $Y = X \cap H$ of $X$, there are  $1 \le
 b_1\le \ldots \le b_{d-1}$ such that $Y=Y_0\cup Y_1\cup \cdots \cup Y_s$,
 where $Y_0=S(b_1,\ldots,b_{d-1})$,  $Y_i=H_i^{m_i}$ is structure of
 multiplicity $m_i$ on $H_i=\PP^{d-1} \subset H$, and:
\begin{enumerate}[(i)] 
\item \label{uno-bis}  $m_1+\cdots+m_s\leq a_{d}$ and $a_1+\dots+a_d=b_1+\dots+b_{d-1}+m_1+\cdots+m_s$,
\item \label{due-bis} $a_j\leq b_j$ for every   $j=1,\dots,d-1$, 
\item \label{tre-bis} If $a_j<b_j$ for some $j\leq d-1$ then let $v=\min\{ j : a_j<b_j\}$. Then $b_j\geq
  a_{j+1}$ for every $j\geq v$,
\item  \label{quattro} the restriction of $\pi$ to $Y_i$ is dominant if and only if $i=0$.
 \end{enumerate}
 Conversely, given  $1 \le b_1\le \ldots \le b_{d-1}$ and
 $m_1,\ldots,m_s$ satisfying the above conditions, there is a
 hyperplane section $Y$ of $X$ whose decomposition takes
 the form $Y=S(b_1,\ldots,b_{d-1}) \cup 
 H_1^{m_1} \cup \cdots \cup H_s^{m_s}$.
\end{theorem}

\begin{proof} 
  We start with the geometric view point. 
  We use the notation of Theorem \ref{geo} and of its proof, only this
  time we do not have $\W=\V/\O_{\PP^1}$ torsionfree. Its locally free part
  $\W_1$ is a direct sum of line bundles of the form
  $\O_{\PP^1}(b_i)$. 
  Its torsion part $\T$ is a direct sum of structure sheaves of
  distinct points $p_1,\ldots,p_s$ of $\PP^1$, taken with
  multiplicities $m_1,\ldots,m_s$. Indeed, we have seen that $\W
  \simeq \W_1 \oplus \T$, and dualizing 
  \[
  0 \to \O_{\PP^1} \to \V \to \W \to 0,
  \]
  we obtain the long exact sequence:
  \[
  0 \to \bigoplus_{i=1}^{d-1}\O_{\PP^1}(-b_i) \to
  \bigoplus_{i=1}^{d}\O_{\PP^1}(-a_i) \to \O_{\PP^1} \to
  \mathcal{E}xt_{\PP^1}^1(\T,\O_{\PP^1}) \to 0,
  \]
  which show that $\T$ has rank $1$ at ever point of its
  support. Therefore:
  \[
  \W \simeq \bigoplus_{i=1}^{d-1}\O_{\PP^1}(b_i) \oplus \bigoplus_{i=1}^{s}\O_{p_i^{m_i}}.
  \]

  Now, $\PP(\W)$ consists of the union of $\PP(\W_1)$ and of
  $\PP(\O_{p_i^{m_i}})$. The tautological linear system over $\PP(\W)$
  maps to
  $Y \subset H$ and sends
  $\PP(\W_1)$ to $Y_0=S(b_1,\ldots,b_{d-1})$, and
  $\PP(\O_{p_i^{m_i}})$ to a structure of multiplicity $m_i$ over the
  image $H_i$ of $\PP(\O_{p_i})$. This gives the required
  decomposition of $Y$.

  Computing the first Chern class  gives condition
  \eqref{uno-bis}. The conditions \eqref{due-bis} and \eqref{tre-bis}
  are proved exactly as we did in Theorem \ref{geo}. 
  Condition  \eqref{quattro} is clear since $\pi$ sends the whole
  component $H_i$ to $p_i$.
  
 For the algebraic point of view we use the notations introduced in the proof of Theorem \ref{basic2}. 
 In this case the form $L=\sum_{i=1}^d f_i s_i$ is  non-zero and factors as $L=gL'$ with $g=\GCD(f_1,\dots, f_d)$. Set $c=\deg g$.  Then $L'=\sum_{i=1}^d f_i's_i$ is irreducible and $\deg f_i'=a_i-c$ (by convention the polynomial $0$ as arbitrary degree). 
 Then  $(f_1,\dots,f_d)=g(f_1',\dots,f_d')$ and hence  the syzygy module of $f_1,\dots,f_d$  is the syzygy module of $f_1',\dots,f_d'$ up to a shift in degrees. Keeping track of the shifts we see that the syzygy module is free generated in  $b_1,\dots, b_{d-1}$  and  
 $\sum_{i=1}^{d-1} b_i+c=\sum_{i=1}^d  a_i$. The map \ref{iso} is still surjective but not injective. Indeed    $(L')/(L)$ is  the kernel of  the map $\Sym_A(F^*) \to \Sym_A(G^*)$. Hence we have an induced surjective $K$-algebra map $R(a_1\dots, a_d)/(L) \to R(b_1\dots, b_d)$  with kernel $(L')\cap R(a_1\dots, a_d)/(L)$ and this gives the irreducible component $Y_0$. The polynomial $g$ factors 
 $g=\ell_1^{m_1}\cdots \ell_s^{m_s}$ with  $\ell_i$ distinct linear forms,  so that $c=\sum_{i=1}^s m_i$.  Each factor $\ell_j^{m_j}$ give a component $Y_j$ of $Y$ corresponding to the quotient ring $R_j=R(a_1\dots, a_d)/J_j$ with $J_j=(\ell_j^{m_j}) \cap R(a_1\dots, a_d)$. To analyze the structure of $R_j$ we may assume $\ell_j=x$ and set $u=m_j$. One has $J_j=\sum_{i=1}^d (x^u A_{v_ia_i-u} s_i^{v_i})$ where $v_i$ is the upper integral  part of $u/a_i$ and  its  radical is $\sum_{i=1}^d (x A_{a_i-u} s_i)$. So 
 $R(a_1\dots, a_d)/\sqrt{J_j}\simeq K[y^{a_1}s_1,\dots, y^{a_d}s_d]$ and the latter is a polynomial ring.  So the reduced structure of $Y_j$ is a $\PP^{d-1}$. That the multiplicity of $Y_j$ is $u$ follows easily form the fact that $K[x,y]/(x^u)$ has multiplicity $u$.

 Now we prove the converse. Assume that numbers $a_1,\dots, a_d$, $b_1,\dots, b_{d-1}$ and $m_1,\dots, m_s$ satisfying the  the three numerical conditions. If $b_j>a_j$ for some $j$ we set Set $c=m_1+\dots+m_s$. We have $a_i=b_i$ for $i<v$ and hence 
 $\sum_{i=v}^d a_i =c+ \sum_{i=v}^{d-1} b_i$. It follows that 
$$a_v=c+\sum_{i=v+1}^d (b_{i-1}-a_i) \geq c.$$
Set $a'=a_v',\dots, a_{d}'$  and $b'=b_v',\dots, b_{d-1}'$ with $a_i'=a_i-c$ and   $b_i'=b_i-c$. 
Then the sequences $a'$ and $b'$ satisfy the conditions of \ref{bseq}  with the only exception of the fact that $a_v'$ can be $0$ while in  \ref{bseq}  it is assumed to be positive. However one can check that the construction given in \ref{bseq}  works also if some of the $a_i$ are $0$. 
Therefore the construction given in  \ref{bseq}  produce homogeneous  polynomials $f_v',\dots, f_d'$ such that $\GCD(f_v',\dots,f_d')=1$, $\deg f_i'=a_i'$ and the syzygy module of  $f_v',\dots,f_d'$ is free with generators in degree $b_v',\dots, b_{d-1}'$. Now we set $f_i'=0$ for $i=1,\dots, v-1$,  $L'=\sum_{i=1}^d f_i's_i$, $g=\ell_1^{m_1}\cdots \ell_s^{m_s}$ with  $\ell_i$ distinct linear forms in $A$ and finally $L=gL'$.  Then keeping track of the shifts one has that the syzygy module of $f_1,\dots, f_d$ is freely generated by elements of degree $b_1,\dots,b_{d-1}$. We have seen on the first part of the proof how the factorization of $L$ determines the decomposition of the hyperplane section $Y$ of $X$ with the hyperplane defined by $L$. So we conclude that $Y$ has the desired decomposition.  

In the remaining case $a_j=b_j$ for all $j=1,\dots, d-1$  one has $a_d=m_1+\cdots m_s$ and we may take  
$L=gs_d$ with $g=\ell_1^{m_1}\cdots \ell_s^{m_s}$ with  $\ell_i$ distinct linear forms in $A$.   
  \end{proof} 
  
  We illustrate the construction   in one example: 
  
  \begin{example} Consider $a=(2,5,7,10)$, $b=(2,7,11)$, $m_1=1$ and $m_2=3$. Then the numerical conditions of \ref{redusect} are satisfied and hence there exists an hyperplane section $Y=X\cap H$ of $X=S(2,5,7,10)$  such that $Y=Y_0\cap Y_1\cap Y_2$ with $Y_0=S(2,7,11)$, $Y_1=\PP^3$ and $Y_2$ a structure of degree $3$ on a $\PP^3$. 
  To describe a linear form $L$ that define $H$ we proceed as described in the proof of \ref{redusect}. Here $v=2$ and $c=m_1+m_2=4$ so that $a'=(a_2-c,a_3-c,a_4-c)=(1,3,6)$ and $b'=(b_2-c,b_3-c)=(3,7)$. Then $f_2',f_3',f_4'$ are defined, up to sign, as the $2$-minors of the matrix 
  
  $$\left(\begin{array}{cc}
  y^2 & 0 \\
  1  & y^4 \\
  0  & x
  \end{array}
 \right)$$
 i.e. $f_2',f_3',f_4'=x, xy^2, y^6$,  and $L'=\sum_{i=2}^4  f_i's_i$ and $g=\ell_1\ell_2^3$ with $\ell_1,\ell_2$ distinct linear forms, for example we may take $\ell_1=x$ and $\ell_2=y$. Then 
 $$L=gL'=xy^3(xs_2+xy^2s_3+y^6s_4)=x^2y^3s_2+x^2y^5s_3+xy^9s_4.$$ 
 The hyperplane $H$ is defined by $L=0$. 
  \end{example}

   \begin{example}  We have seen \ref{gengen}   there are  $16$ different $3$-dimensional scrolls that appear  as irreducible hyperplane section of  $X=S(4,5,6,9)$. Obviously they have all degree $24$. According to \ref{redusect}
 we may list all the   $3$-dimensional scrolls that appear as irreducible components of  reducible hyperplane  
 sections of $X$. There are $71$ such scrolls, they are  described in the following table where the first column denotes the degree  and the second the number of different scrolls of that degree.  
 
 $$\begin{array}{lll}
 23 & 13 & [4, 5, 14],[4, 6, 13],[4, 7, 12],[4, 8, 11],[4, 9, 10],[5, 6, 12],[5, 7, 11],\\
      &       &[5, 8, 10],[5, 9, 9],[6, 6, 11],[6, 7, 10],[6, 8, 9],[7, 7, 9],\\
      \hline 
22  & 10 &   [4, 5, 13],[4, 6, 12],[4, 7, 11],[4, 8, 10],[4, 9, 9],[5, 6, 11],[5, 7, 10],\\
     &       &   [5, 8, 9],[6, 6, 10],[6, 7, 9], \\
         \hline  
21 & 7 &  [4, 5, 12],[4, 6, 11],[4, 7, 10],[4, 8, 9],[5, 6, 10],[5, 7, 9],[6, 6, 9],\\
    \hline  
 20 & 4 & [4, 5, 11],[4, 6, 10],[4, 7, 9],[5, 6, 9],\\
     \hline  
 19 & 2 & [4, 5, 10],[4, 6, 9],\\
     \hline  
  18 & 1 & [4, 5, 9],\\
      \hline 
15 & 1 & [4, 5, 6],
\end{array}
$$

 \end{example} 
\subsection{Cones}

 We mentioned in the introduction that a singular irreducible
 variety $X$ of minimal degree
 is cone over a non-singular one. What happens to the hyperplane
 sections of $X$ if $X$ is indeed
 a cone?  Geometrically, this amounts to allow some of the $a_i$ to
 vanish.
 Let us call $X^0$ the base of the
 cone, which is a smooth variety of minimal degree sitting in a linear
 subspace $\PP^{n_0} \subset \PP^n$, defined by the vanishing of $n-n_0$
 linear forms. We may assume that these forms are $x_{n_0+1},\ldots,x_n$. Given a set of defining equations
 of $X^0$ in $\PP^{n_0}$ in the variables $x_0,\ldots,x_{n_0}$, the variety $X$ is defined in $\PP^n$ by the
 same set of equations, seen as equations in the variables  $x_0,\ldots,x_{n}$.
 This expresses the fact $X$ is a cone over $X_0$, the vertex being
 the subspace $M\subset \PP^n$ of codimension $n_0+1$ defined by the
 vanishing of $x_0,\ldots,x_{n_0}$.
 We may replace $\PP^{n_0}$ with any other linear subspace of
 dimension $n_0$ disjoint
 from $M$ to obtain an equivalent description.

 \begin{remark}
   A hyperplane section $Y$ of $X$ is isomorphic to:
   \begin{enumerate}[(i)]
   \item  \label{conovuoto} $X^0$ if $M \cap H = \emptyset$; 
   \item \label{cono} a cone over $X^0$ with vertex $M\cap H$ if $M \cap H \ne \emptyset$ and $M \not\subset H$; 
   \item \label{riconduce} a cone over $X^0 \cap H$ with vertex $M$ if $M \subset H$.
   \end{enumerate}
 \end{remark}

 \begin{proof}
   We have $M \subset H$ if and only if the defining equation of $M$
   depends on the variables $x_0,\ldots,x_{n_0}$ only, in which case the
   ideal of $Y$ is the ideal of $Y^0 = H \cap X^0$ in the variables
   $x_0,\ldots,x_{n_0}$, seen as an ideal of $K[x_0,\ldots,x_n]$, i. e.
   $Y$ is a cone over $Y^0$ with vertex on $M$.

   If $M \not\subset H$, then $H$ is spanned by $H\cap M$ and a
   subspace of dimension $n_0$, disjoint from $M$. Choosing this space
   as our $\PP^{n_0}$, we get that $Y$ is the cone over $X^0$ with
   vertex $M\cap H$, which gives cases \eqref{cono} and
   \eqref{conovuoto}, according to whether $M$ is positive-dimensional
   or consists of a single point.
 \end{proof}

This settles the situation for cones, as clearly one of the three
cases must occur, and, in case \eqref{riconduce}, the
isomorphism type of $X^0 \cap H$ is controlled by Theorem \ref{basic2}.


\begin{thebibliography}{99}

\bibitem[B]{B}  J. Brawner,   \emph{Tetragonal curves, scrolls, and K3 surfaces},    Trans. Amer. Math. Soc. 349 (1997) , 3075--3091.
   

\bibitem[BH]{BH} W. Bruns, J. Herzog, \emph{Cohen-Macaulay rings},
Cambridge Studies in Advanced Mathematics 39, Cambridge University Press,
Cambridge, 1993.
 
 
\bibitem[E]{E} D. Eisenbud,
{\em Commutative Algebra with a View Toward Algebraic Geometry},
Graduate Texts in Mathematics, 150. Springer-Verlag, New York, 1995.

\bibitem[EH]{EH}  D. Eisenbud, J. Harris,
\emph{On varieties of minimal degree (a centennial account)},
Algebraic geometry, Bowdoin, 1985 (Brunswick, Maine, 1985), 3--13,
Proc. Sympos. Pure Math., 46, Part 1,
Amer. Math. Soc., Providence, RI, 1987.

\bibitem[F]{F}  R. Fr\"oberg,  \emph{An inequality for Hilbert series of graded algebras},  Math. Scand. 56 (1985), no. 2, 117--144.
 
\bibitem[H]{H}  J. Harris,  \emph{Algebraic Geometry: a First Course}, Graduate Texts in Mathematics, Springer-Verlag, 1992.
 
 \end{thebibliography}
\end{document}